\documentclass[12pt]{article}

\usepackage{epsfig}
\usepackage{latexsym}
\usepackage{caption}
\usepackage{amsfonts}

\usepackage{amssymb}
\usepackage{mathrsfs}
\usepackage{amsmath}

\usepackage{graphicx}
\usepackage{float}
\usepackage{pict2e}
\usepackage{enumerate}
\usepackage{enumitem}
\usepackage{authblk}

\usepackage{graphics}
\usepackage{tikz}
\usepackage{ulem}
\usepackage{cases}

\usepackage{tikz-3dplot}
\usepackage{subcaption}

\usepackage[pdfauthor={derajan},pdftitle={How to do this},pdfstartview=XYZ,bookmarks=true,
colorlinks=true,linkcolor=blue,urlcolor=blue,citecolor=blue,pdftex,bookmarks=true,linktocpage=true,hyperindex=true]{hyperref}

\usepackage{color}

\usepackage[colorinlistoftodos,prependcaption,textsize=scriptsize, textwidth=24mm,]{todonotes}

\usetikzlibrary{patterns}

\makeatother

\newtheorem{theorem}{Theorem}

\newtheorem{conjecture}[theorem]{Conjecture}

\newtheorem{claim}{Claim}

\newenvironment{proof}{\noindent {\bf
		Proof.}}{\rule{2.5mm}{2.5mm}\par\medskip}

\title{\large{\bf Partitioning triangle-free planar graphs into a forest and a linear forest}}

\author{
	Guanwu Liu\thanks{School of Mathematics and Statistics, Shandong University of Technology, Zibo, Shandong, 255000, China. 
		Research supported by the National Natural Science Foundation of China (No.~12401459) and by the Natural Science Foundation of Shandong Province of China (ZR2024QA075). 
		Email: \texttt{liuguanwu@hotmail.com}} 
	\quad 
	Rongxing Xu\thanks{School of Mathematical Sciences, Zhejiang Normal University, Jinhua, Zhejiang, 321000, China. Research supported by the National Natural Science Foundation for Young Scientists of China (Grant No.~12401472), and the Zhejiang Provincial Natural Science Foundation of China (Grant No.~LQN25A010011). 
		Email: \texttt{xurongxing@zjnu.edu.cn}}
}

\begin{document}
	\maketitle
	
	\begin{abstract}
		Raspaud and Wang conjectured that every triangle-free planar graph can be vertex-partitioned into an independent set and a forest. Independently, Kawarabayashi and Thomassen also remarked that this might be true, after providing another proof of a result of Borodin and Glebov, showing this result for planar graphs of girth~5. Subsequently, Dross, Montassier, and Pinlou raised the same question and proved that every triangle-free planar graph can be partitioned into a forest and another forest of maximum degree~5. More recently, Feghali and \v{S}\'{a}mal improved this bound on the maximum degree to~3. In this note, we further improve the result by showing that every triangle-free planar graph can be partitioned into a forest and a linear forest, that is, a forest of maximum degree~2.
	\end{abstract}
	
	\section{Introduction}

	All graphs considered in this paper are finite and simple.
	A \emph{linear forest} is a forest in which every connected component is a path.
	We say that a graph $G$ can be partitioned into graphs $H_1, H_2, \ldots, H_k$ if there exists a partition $(V_1, V_2, \ldots, V_k)$ of $V(G)$ such that each $V_i$ induces a subgraph of $G$ isomorphic to $H_i$.
	
	The problem of partitioning a graph into subgraphs with simpler structure has attracted considerable attention in graph theory.
	Typical examples of such structures include independent sets (corresponding to proper colorings), forests (related to the vertex arboricity of a graph \cite{CKW1968}), linear forests, and, more generally, subgraphs of bounded degeneracy.
	
	The celebrated Four-Color Theorem \cite{AH1977,AHK1977} asserts that every planar graph can be partitioned into four independent sets. Thomassen \cite{Thomassen2001} proved that every planar graph can be partitioned into an independent set and a 3-degenerate graph. 
	Chartrand, Kronk and Wall \cite{CKW1968} proved that every planar graph can be partitioned into three forests (it follows easily from the fact that every planar graph is $5$-degenerate). This result was subsequently strengthened in several directions. Poh \cite{Poh1990} proved that every planar graph can be partitioned into three linear forests.
	Thomassen \cite{Thomassen1995} showed that the vertex set of every planar graph can be partitioned into a forest and a 2-degenerate graph, note that every 2-degenerate graph can be partitioned into two forests (by greedy algorithm).
	Borodin \cite{Borodin1979} proved that every planar graph admits an acyclic coloring with at most five colors, where an acyclic coloring is a proper coloring in which every two color classes induce a forest; this immediately implies that every planar graph can be partitioned into an independent set and two forests.
	
	For triangle-free planar graphs, Gr\"{o}tzsch's theorem \cite{Grotzsch} asserts that every such graph is 3-colorable, that is, it can be partitioned into three independent sets.
	Moreover, by Euler’s formula, every triangle-free planar graph has a vertex of degree at most three, and it follows easily that every triangle-free planar graph can be partitioned into two forests.
	Raspaud and Wang \cite{RW2008} conjectured that every triangle-free planar graph can be partitioned into an independent set and a forest.
	Independently, Kawarabayashi and Thomassen \cite{KT2009} also remarked that this might be true, after providing another proof of an earlier result of Borodin and Glebov \cite{BG2001}, who showed that every planar graph of girth at least five can be partitioned into an independent set and a forest.
	
	\begin{conjecture}[\cite{KT2009,RW2008}]
		\label{Conj-IF-triangle}
		Every triangle-free planar graph can be partitioned into an independent set and a forest.
	\end{conjecture}
	
	Conjecture~\ref{Conj-IF-triangle}, if true, would imply many other results such as Gr\"{o}tzsch's theorem, and would give positive answers to Question~28 in \cite{Zhu2020} and Question~2 in \cite{CY2017}, which concern list colorings of triangle-free planar graphs with restrictions on the list assignments.
	
	Subsequently, Dross, Montassier, and Pinlou \cite{DMP2017} posed exactly the same problem, and further asked for the minimum integer $d$ such that every triangle-free planar graph can be partitioned into a forest and a forest of maximum degree $d$.
	In the same paper, they showed that $d \leq 5$.
	Note that an independent set is a forest of maximum degree~0.
	It is also worth mentioning that Montassier and Ochem \cite{MO2015} proved that not every triangle-free planar graph can be partitioned into two graphs of bounded degree. 
	Recently, Feghali and \v{S}\'{a}mal \cite{FS2024} improved the bound by showing that $d \leq 3$ using list-coloring techniques.
	In this paper, we further improve this to $d \leq 2$.
	
	\begin{theorem}
		\label{main-thm0}
		Every triangle-free planar graph can be partitioned into a forest and a linear forest.
	\end{theorem}
	
	We now fix some notation to be used throughout the paper.    
	Let $G$ be a graph and let $\phi: V(G) \to \{1,2,\ldots, k\}$ be a (not necessarily proper) coloring of $V(G)$.  
	For $i \in \{1,2,\ldots, k\}$, we write $\phi^{-1}_{G}(i)$ for the set of vertices of $G$ colored $i$ by $\phi$.  
	When the underlying graph is clear from the context, we simply write $\phi^{-1}(i)$.  
	For a subset $U \subseteq V(G)$, we use $\phi|_{U}$ to denote the restriction of $\phi$ to $U$, and $G[U]$ to denote the subgraph of $G$ induced by $U$, while $G-U$ denotes the subgraph of $G$ induced by $V(G)\setminus U$. For two colorings $\phi$ and $\psi$, we write $\phi \cup \psi$ for their union, whenever they are consistent on the intersection of their domains.
	If $H$ is a subgraph of $G$, then $G \setminus E(H)$ denotes the graph obtained from $G$ by deleting all edges of $H$ but keeping all vertices.  
	If $G$ is a planar graph, we denote by $B_G$ the boundary cycle of the outer face of $G$.  
	For a vertex $u \in V(B_G)$, a vertex $v \in V(B_G)$ is called a \emph{boundary neighbor} of $u$ if $uv \in E(B_G)$.

	\section{The proof of Theorem \ref{main-thm0}}
	
	Instead of proving Theorem~\ref{main-thm0} directly, we establish the following stronger result, which is inspired by the main idea of \cite{FS2024} but stated in a stronger form.
	
	\begin{theorem}\label{main-thm}
		Let $G$ be a triangle-free plane graph, and let $z$ be a boundary vertex of $G$.
		Suppose $P$ and $Q$ are disjoint (possibly empty) vertex sets on the boundary of $G$ such that
		$z \notin P\cup Q$, $|P|\leq 3$, the vertices of $P$ appear consecutively on the boundary of $G$,
		and $Q$ is an independent set.
		Let $\eta: P \to \{1,2\}$ be a precoloring of $P$ such that if $|P|=3$, then $\eta$ assigns color $1$ to exactly one endpoint of $G[P]$ and color $2$ to the other two vertices.
		Then there exists a coloring $\phi: V(G) \to \{1,2\}$ such that the following hold.
		\begin{enumerate}[label={(C\arabic*)},leftmargin=1.3cm]
			\item \label{C1} $\phi(u) = \eta(u)$ for each $u \in P$ and $\phi(u)=2$ for each $u \in Q$.
			\item \label{C2} The vertices of color $1$ induce a linear forest, and the vertices of color $2$ induce a forest.  
			\item \label{C3} If $\eta(u) = 1$ for some $u \in P$, then for any $v \in N_{G}(u) \setminus P$, $\phi(v)=2$.		
			\item \label{C4} If $|P| \leq 2$, $\eta(u) = 1$ for exactly one vertex $u \in P$ and $v \in N_{G}(u)\cap Q$, then for every boundary neighbor $w \neq v$ of $u$, there is no $\phi$-monochromatic path joining $w$ and $v$.
			\item \label{C5} If $\phi(z)=1$, then there is at most one neighbor of $z$ which is also colored $1$ in $\phi$.
		\end{enumerate}
		A coloring of $G$ satisfying \ref{C1}-\ref{C5} is said {\rm valid} for $(G,P,Q,z,\eta)$.
	\end{theorem}

	\begin{proof}
		Suppose the theorem is false. Let $G$ be a counterexample with $|V(G)|$ as small as possible and, subject to this, with $|Q|$ as large as possible. Clearly $G$ must be connected; otherwise, by the minimality of $G$, each component could be colored separately.
		
		\begin{claim}\label{claim-two-connected}
			$G$ is $2$-connected.
		\end{claim}  
		\begin{proof}
			Assume to the contrary that $G$ is not $2$-connected, and let $x$ be a cut vertex. Then $G$ can be written as the union of two induced subgraphs $G_1$ and $G_2$ with $V(G_1)\cap V(G_2)=\{x\}$ and $V(G_1)\cup V(G_2)=V(G)$. Without loss of generality, assume that $|P \cap V(G_1)| \geq |P \cap V(G_2)|$. Set $P_1 = P \cap V(G_1)$ and 
			\[
			P_2 =
			\begin{cases}
				\{x\}, & \text{if } P \cap V(G_2) = \emptyset, \\[6pt]
				P \cap V(G_2), & \text{otherwise}.
			\end{cases}
			\]
			We define two vertices $z_1$ and $z_2$ as follows. If $z\in V(G_1)$, then set $z_1 = z$ and choose $z_2 \in V(B_{G_2}) \setminus (P_2 \cup Q_2)$. If $z \notin V(G_1)$, then set $z_2 = z$ and choose $z_1 \in V(B_{G_1}) \setminus P_1$, making sure that $z_1 \notin Q$ unless $V(B_{G_1}) \setminus (P \cup Q) = \emptyset$.
			Let $Q_1=(Q \cap V(G_1)) \setminus \{z_1\}$ and $Q_2 = Q \setminus (Q \cap V(G_1))$.
			
			By the minimality of $G$, there exists a coloring $\phi_1$ valid for  $(G_1, P_1, Q_1, z_1, \eta|_{P_1})$ and a coloring $\phi_2$ valid for $(G_2, P_2, Q_2, z_2, (\eta \cup \phi_1)|_{P_2})$. Note that if $\phi_1(x)=1$, then $\phi_2(v)=2$ for every $v \in N_{G_2}(x)$. 
			In most cases, it is straightforward to check that $\phi_1 \cup \phi_2$ is valid 
			for $(G,P,Q,z,\eta)$. We therefore omit the details of this verification and only 
			show that $\phi(z_1)=2$ in the case where $V(B_{G_1}) \setminus (P \cup Q) = \emptyset$ 
			and $z_1 \in Q$. This case only happens when $B_{G_1}$ is a cycle containing four vertices, three of which are in $P$ (hence we do not need to verify \ref{C4}), and the rest one is in $Q$. Without loss of generality, assume that $B_{G_1}=v_1v_2v_3v_4v_1$, $P=\{v_1,v_2,v_3\}$ and $v_4 \in Q$. So $z_1=v_4$ by our choice. By the condition of $\eta$, either $\eta(v_1)=1$ or $\eta(v_3)=1$, in either case we will have $\phi(z_1) = \phi(v_4)=2$ by \ref{C3} as $v_4$ is the neighbor of both $v_1$ and $v_3$.
		\end{proof}
		
		By Claim~\ref{claim-two-connected}, we have that $B_G$ is a cycle. We then have the following claim.
		
		\begin{claim}\label{claim-chord}
			$B_G$ has no chord unless $|P|=3$, in which case the chord must be incident with the middle vertex of $G[P]$.
		\end{claim}
		\begin{proof}
			Assume to the contrary that $xy$ is a chord of $B_G$. If $|P|=3$,  assume further that neither $x$ nor $y$ is the middle vertex of $G[P]$. Let $G_1$ and $G_2$ be induced subgraphs of $G$ with $V(G_1)\cap V(G_2)=\{x,y\}$ and $V(G_1)\cup V(G_2)=V(G)$. By the assumption on $xy$, we may assume  without loss of generality that $P \subseteq V(G_1)$. Set $P_2=\{x,y\}$ and define $Q_1,Q_2,z_1,z_2$ as in the proof of Claim~\ref{claim-two-connected}.

			By the minimality of $G$, there exists a coloring $\phi_1$ valid for  $(G_1, P, Q_1, z_1, \eta)$ and a coloring $\phi_2$ valid for $(G_2, P_2, Q_2, z_2, \phi_1|_{P_2})$. We shall verify that $\phi=\phi_1 \cup \phi_2$ is valid for $(G,P,Q,z,\eta)$.
			
			First, for each $u \in P$, since $P \subsetneq V(G_1)$ and $\phi_1$ satisfies \ref{C1}, we have $\phi(u)=\phi_1(u)=\eta(u)$. For a vertex $v \in Q$, if $v \in Q_1 \cup Q_2$, then clearly $\phi(v)=2$ by \ref{C1} for both $\phi_1$ and $\phi_2$. 
			Otherwise $v \notin Q_1 \cup Q_2$, which can occur only when $V(B_{G_1}) = P \cup \{v\}$ and $v \in Q$ (with $z \notin V(G_1)$), in which case we set $z_1 = v$. 
			By the same argument as in Claim~\ref{claim-two-connected}, we then have $\phi(v)=\phi(z_1)=2$. 
			Thus $\phi$ satisfies \ref{C1}. 
			
			Next, since $\phi_1$ and $\phi_2$ satisfy \ref{C2}, the sets $\phi_1^{-1}(1)$ and $\phi_2^{-1}(1)$ induce linear forests, while $\phi_1^{-1}(2)$ and $\phi_2^{-1}(2)$ induce forests. Moreover, as $\phi_2$ satisfies \ref{C3}, there are no edges between $\phi_1^{-1}(1)$ and $\phi_2^{-1}(1)\setminus V(G_1)$. 
			Hence $\phi^{-1}(1)=\phi_1^{-1}(1)\cup \phi_2^{-1}(1)$ also induces a linear forest. 
			On the other hand, $\phi^{-1}(2)=\phi_1^{-1}(2)\cup \phi_2^{-1}(2)$ induces a forest, since the two forests $G[\phi_1^{-1}(2)]$ and $G[\phi_2^{-1}(2)]$ intersect in a path.

			Suppose $u \in P$ and $\phi(u)=1$. Let $v \in N_G(u)\setminus P$. If $v \in V(G_1)$, then $\phi(v)=\phi_1(v)=2$ as $\phi_1$ satisfies \ref{C3}. Otherwise $v \in V(G_2)\setminus V(G_1)$, implying $u \in P_2=\{x,y\}$, hence $\phi(v)=\phi_2(v)=2$ as $\phi_2$ also satisfies \ref{C3}. Therefore, $\phi$ satisfies \ref{C3}.
			
			Suppose, to the contrary, that \ref{C4} fails for some $u \in P$ with $\eta(u)=1$ where $|P|\leq 2$.  Then there exists a $\phi$-monochromatic path $R$ of color $2$ joining $w$ and $v$, where $v \in N_G(u)\cap Q$ and $w\neq v$ is a boundary neighbor of $u$. Let $R_i = R\cap G_i$ for $i=1,2$. 
			Since both $\phi_1$ and $\phi_2$ satisfy \ref{C4}, $R$ is neither contained entirely in $G_1$ nor in $G_2$, so $V(R_i)\setminus\{x,y\}\neq\emptyset$ for each $i$. Thus if $\{w,v\}\subsetneq V(G_1)$, then $\{x,y\}\subsetneq V(R)$, implying $\phi_1(x)=\phi_1(y)=2$. Hence $[R_1\setminus E(R_2)]\cup xy$ is a $\phi_1$-monochromatic path of color $2$ joining $w$ and $v$, a contradiction. Similarly, it is impossible that $\{w,v\}\subsetneq V(G_2)$. Thus assume $v\in V(G_j)$ for some $j \in\{1,2\}$ and $w\in V(G_{3-j})$, which forces $u\in\{x,y\}$. 
			Without loss of generality, assume $u=y$. Hence $\eta(y)=1$, and therefore $x\in V(R_1)\cap V(R_2)$ with $\phi_1(x)=\phi_2(x)=2$. Consequently, the path $R_j$ is a 
			$\phi_j$-monochromatic path of color $2$ joining $v$ and $x$, which implies 
			that $\phi_j$ violates \ref{C4}, since $x$ is a boundary neighbor of $u$ in $G_j$. 
			This is a contradiction. Therefore, \ref{C4} holds.

			Finally, it is straightforward to verify that $\phi$ satisfies \ref{C5}, since $z_1=z$ if $z\in V(G_1)$ and $z_2=z$ otherwise, and both $\phi_1$ and $\phi_2$ satisfy \ref{C5}.
			
			This completes the proof of this claim.
		\end{proof}
		
		\begin{claim}\label{claim-2-chord} 
			For each internal vertex $s$ of $G$, if $s$ has two neighbors $x,y\in V(B_G)$, then neither $x$ nor $y$ belongs to $Q$, unless $|P|=3$, in which case one of $x,y$ lies in $Q$ and the other is the middle vertex of $G[P]$.
		\end{claim}
		
		\begin{proof}
			Suppose to the contrary, and without loss of generality, $x \in Q$ and $y$ is not the middle vertex of $G[P]$. The path $xsy$ divides $G$ to two induced graph $G_1$ and $G_2$ with $V(G_1)\cap V(G_2) =\{x,s,y\}$ and $V(G_1) \cup V(G_2)= V(G)$. Without loss of generality, we may assume that $P \in V(G_1)$. Let $Q_1=Q \cap V(G_1)$ and  $Q_2 = Q \cap (V(G_2) \setminus \{y\})$.
			We define two vertices $z_1$ and $z_2$ as follows. If $z\in V(G_1)$, then set $z_1 = z$ and choose $z_2 \in V(B_{G_2}) \setminus (Q_2 \cup \{s, y\})$. If $z \notin V(G_1)$, then set $z_1=s$ and $z_2 = z$.
			
			By the minimality of $G$, there is a coloring $\phi_1$ valid for $(G_1,P, Q_1, z_1,\eta)$ and a coloring $\phi_2$ valid for   $(G_2,\{s,y\},Q_2, z_2, \phi_1|_{\{s,y\}})$. We shall verify that $\phi=\phi_1 \cup \phi_2$ is valid for $(G,P,Q,z,\eta)$. The verification of $\phi$ satisfying \ref{C1}, \ref{C3} and \ref{C5} is similar as in the proof of Claim~\ref{claim-chord}, we do not repeat again here and only focus on \ref{C2} and \ref{C4}.

			For \ref{C2}, it is clear that the vertices of color $1$ induce a linear forest.
			If the two forests $G[\phi_1^{-1}(2)]$ and $G[\phi_2^{-1}(2)]$ intersect in a path, then $\phi^{-1}(2)=\phi_1^{-1}(2)\cup \phi_2^{-1}(2)$ also induces a forest.
			Hence, if there were a cycle $C$ colored $2$ in $\phi$, it would follow that $\phi_1(s)=1$ and that $C \cap G_2$ is a $\phi_2$-monochromatic path of color $2$ joining $x$ and $y$.
			This contradicts \ref{C4} for $\phi_2$.
			
			Suppose to the contrary that \ref{C4} does not hold for some $u \in P$ with $\eta(u)=1$. 
			Then there exists a $\phi$-monochromatic path $R$ of color~2 joining $w$ and $v$, 
			where $v \in N_G(u)\cap Q$ and $w\neq v$ is a boundary neighbor of $u$. 
			Let $R_i=R\cap G_i$ for $i=1,2$. 
			Since both $\phi_1$ and $\phi_2$ satisfy \ref{C4}, the path $R$ is neither fully contained in $G_1$ nor in $G_2$.
			
			If $\phi_1(y)=1$ and $\phi_1(s)=2$, then by arguments analogous to those in the 
			proof of Claim~\ref{claim-chord} (with $s$ and $x$ playing the roles of $y$ 
			and $x$ therein, respectively), we arrive at a contradiction.

			If $\phi_1(y)= \phi_1(s) = 1$, then $v$ and $w$ must belong to distinct subgraphs, one in $G_1$ and the other in $G_2$. This implies that $u=y$, hence $y \in P$. But in this case $\phi_1(s)=2$ by \ref{C3}, a contradiction.
			
			Thus we have $\phi_1(y)=2$, which implies $u \neq y$, and hence ${v,w}\subsetneq V(G_1)$.
			It follows that $\phi_1(s)=1$, for otherwise $R_1$ would be a $\phi_1$-monochromatic path of color~2 joining $w$ and $v$.
			Therefore $y \in V(R)$, and consequently $R_2$ is a $\phi_2$-monochromatic path of color~2 joining $x$ and $y$.
			This contradicts \ref{C4}, since $x \in Q_2$, $\phi_2(s)=1$, and $y$ is a boundary neighbor of $s$ (recall that ${s,y}$ served as the set $P$ in the construction of $\phi_2$).
			Hence \ref{C4} holds.
		\end{proof}

		\begin{claim}\label{claim:4-cycle}
			$G$ has no separating $4$-cycles and $B_G$ contains at least $5$ vertices.
		\end{claim}
		\begin{proof}
			Assume to the contary $C=u_1u_2u_3u_4u_1$ is a separating $4$-cycle of $G$. Let $G_1$ (resp.\ $G_2$) be the subgraph induced by $C$ together with all vertices and edges in the exterior (resp.\ interior) of $C$. 
			By the minimality of $G$, there is a coloring $\phi_1$ of $G_1$ valid for $(G_1, P, Q \cap V(G_1),z, \eta)$. By \ref{C2} for $\phi_1$, $C$ is not monochromatic, so without loss of generality, we may assume that $\phi_1(u_1) = 1$ and $\phi_1(u_2) = 2$. Let $G_2' = G_2 - \{u_1\}$, and let  
			\[
			P' = \begin{cases}
				\{u_3\}, & \text{if } \phi_1(u_4)=2, \\
				\{u_3,u_4\}, & \text{otherwise,}
			\end{cases} \text{ and }
			Q' = \begin{cases}
				N_{G_2}(u_1), & \text{if } \phi_1(u_4)=2, \\
				N_{G_2}(u_1) \setminus \{u_4\}, & \text{otherwise.}
			\end{cases}
			\]
			Since $G$ has no triangles, $Q'$ is an independent set in both cases.
			
			By the minimality of $G$, $G_2'$ has a coloring $\phi_2$ valid for $(G'_2, P', Q', z',\phi_1|_{P'})$, where $z'\in V(B_{G'_2})\setminus (P'\cup Q')$. We shall show that $\phi = \phi_1 \cup \phi_2$ is valid for $(G,P,Q, z,\eta)$. It is straightforward to check that $\phi$ satisfies \ref{C1}, \ref{C3}, and \ref{C5}. Moreover, by analogous arguments (with $u_3$ playing the role of $s$) as in the proof of Claim~\ref{claim-2-chord}, one can also verify that \ref{C4} holds for $\phi$. Therefore, it remains to check \ref{C2}.
			
			Note that there are no edges between the vertices of color~1 on $C$ and the vertices of color~1 in $V(G_2)\setminus V(C)$. 
			Indeed, for each $v \in N_{G_2'}(u_1)\setminus V(C)$ we have $v \in Q'$, and hence $\phi(v)=\phi_2(v)=2$ by \ref{C1} for $\phi_2$. 
			If $\phi_1(u_i)=1$ for some $i \in \{3,4\}$, then $u_i \in P'$ by our choice, so for each $v \in N_{G_2'}(u_i)\setminus V(C)$ we also have $\phi(v)=\phi_2(v)=2$ by \ref{C3} for $\phi_2$. 
			Thus the vertices of color~1 under $\phi$ induce a linear forest, since by \ref{C2} both $\phi_1^{-1}(1)$ and $\phi_2^{-1}(1)$ induce linear forests. 
			On the other hand, by analogous arguments (with $u_3$ playing the role of $s$) as in the proof of Claim~\ref{claim-chord}, the vertices of color~2 induce a forest.  
			Hence \ref{C2} holds for $\phi$.
			
			Now, we show the second part of the statement.
			
			Assume to the contrary that $B_G=uvwxu$ is a $4$-cycle. 
			Without loss of generality, we consider the following cases and first extend $\eta$ to a precoloring $\phi_1$ of $B_G$ as follows.
			\begin{itemize}
				\item $P=\{u,x,w\}$ with $\eta(u)=1$ (implying that $\eta(x) = \eta(w) =2$). We color $v$ with $2$.
				\item $P=\{u,x\}$ with $\eta(u)=1$, $\eta(x)=2$, or $P=\{u\}$ with $\eta(u)=1$.
				If $w\in Q$, then $\{v, x\} \cap Q=\emptyset$, we color all remaining uncolored vertices of $B_G$ with $2$. Otherwise, $w \notin Q$, 
				we color $w$ with $1$ and all remaining uncolored vertices of $B_G$ with $2$.
				\item $P=\{u,x\}$ with $\eta(u)=1$, $\eta(x)=1$. 
				We color every vertex in $V(B_G)\setminus P$ with $2$.
				\item Every vertex of $P$ is colored $2$ in $\eta$ and $z=u$. 
				We first color $z$ with $1$, and then color every vertex in $V(B_G)\setminus (P\cup\{z\})$ with $2$.
			\end{itemize} 
			In each case, we relabel the vertices of $B_G$ so that $u_1=u$, $u_2=v$, $u_3=w$, and $u_4=x$. 
			Then, by the same procedure as in the proof of the first part, we can extend $\phi_1$ to a coloring of $G$ that is valid for $(G,P,Q,z,\eta)$, a contradiction. 
		\end{proof} 
		
		In the remainder of this paper, we assume that $B_G=u_0u_1\cdots u_ku_0$, and the indices are taken modulo $k+1$. By Claim~\ref{claim:4-cycle}, we have $k \geq 4$.
		
		\begin{claim}\label{z1}
			\label{claim:BBB}
			There is no $i \in \{0, 1, \ldots, k\}$ such that $\{u_{i-1}, u_i, u_{i+1}\} \cap (P \cup Q) = \emptyset$. 
		\end{claim} 
		\begin{proof}  
			Without loss of generality, assume to the contrary that none of $\{u_0,u_1,u_2\}$ belongs to $P\cup Q$. 
			Note that any valid coloring of $(G,P,Q',z,\eta)$ is also valid for $(G,P,Q,z,\eta)$ whenever $Q\subsetneq Q'$. 
			Thus, by the maximality of $|Q|$, we must have $u_1=z$; otherwise we could add $u_1$ to $Q$, and the resulting set would still be independent by Claim~\ref{claim-chord}. 
			Again, by the maximality of $|Q|$, we have $u_3\in Q$, since otherwise we could add $u_2$ to $Q$. 
			
			Now let $G'=G-\{u_2\}$ and define $Q'=(Q\cup N_G(u_2))\setminus P$. 
			Note that $u_2$ may have a neighbor in $P$, but such a vertex must be the middle vertex of $G[P]$ and is colored $2$ in $\eta$.  
			By Claim~\ref{claim-2-chord} and the fact that $N_G(u_2)$ is an independent set (as $G$ has no triangles), $Q'$ is an independent set. 
			Choose a vertex $z'\in V(B_{G'})\setminus (P\cup Q')$. 
			By the minimality of $G$, there exists a coloring $\phi'$ valid for $(G',P,Q',z',\eta)$. 
			Since $z=u_1\in Q'$, we have $\phi'(z)=2$. 
			Finally, extend $\phi'$ to a coloring $\phi$ of $G$ by setting $\phi(u_2)=1$. 
			This yields a valid coloring of $(G,P,Q,z,\eta)$, a contradiction. 
		\end{proof}
		
		\begin{claim}
			\label{claim-QBQ} 
			There is no $i \in \{0, 1, \ldots, k\}$ such that $u_i \notin P \cup Q$ but $ \{u_{i-1}, u_{i+1}\} \subseteq Q \cup \{v\in P \mid \eta(v)=2\}$.
		\end{claim}
		\begin{proof}  
			Without loss of generality, assume to the contrary that $u_1\notin P\cup Q$, and for each $j\in\{0,2\}$ we have either $u_j\in Q$, or $u_j\in P$ with $\eta(u_j)=2$. 
			Let $G' = G - \{u_1\}$ and define $Q' = (Q \cup N_G(u_1)) \setminus P$. 
			By Claim~\ref{claim-2-chord} and the fact that $N_G(u_1)$ is an independent set, 
			it follows that $Q'$ is also an independent set. 
			Set $z'=z$ if $u_1\neq z$, and otherwise let $z'$ be an arbitrary vertex of $V(B_{G'})\setminus (P\cup Q')$. 
			By the minimality of $G$, the graph $G'$ admits a coloring $\phi'$ valid for $(G',P,Q',z',\eta)$. 
			We then extend $\phi'$ to a coloring $\phi$ of $G$ by setting $\phi(u_1)=1$.  
			Since this adds only a single isolated vertex to the induced forest of color~1, the coloring $\phi$ is valid for $(G,P,Q,z,\eta)$, a contradiction. 
		\end{proof}
		
		\begin{claim}
			\label{claim-QBBQ} 
			There is no $i \in \{0, 1, \ldots, k\}$ such that $\{u_i, u_{i+1}\} \cap (P \cup Q) = \emptyset$  but $ \{u_{i-1}, u_{i+2}\} \subseteq Q \cup \{v\in P \mid \eta(v)=2\}$, and $z\in \{u_i,u_{i+1}\}$ or $z \in N_{G}(v)$ for some $v \in P$ with $\eta(v)=1$. 
		\end{claim}
		\begin{proof}  
			Without loss of generality, assume to the contrary that statement is not hold for $i=1$.  So   neither $u_1$ nor $u_2$ lies in $P\cup Q$, and for each $j \in\{0, 3\}$ we have either $u_j\in Q$ or $u_j\in P$ with $\eta(u_j)=2$. We also have that $z \in \{u_1, u_2\}$ or $z$ is the neighbor of some $v \in P$ with $\eta(v)=1$. Since at least one of $\{u_1,u_2\}$ is not $z$, we may further assume that $z \neq u_2$.  
			Let $f$ be the non-boundary face incident with both $u_1$ and $u_{2}$, and let $x$ and $y$ be the neighbors of $u_1$ and $u_2$, respectively, incident with $f$.

			Let $G' = G - \{u_1,u_2\}$ and  define $Q' = (Q \cup N_G(u_1)\cup N_G(u_2))\setminus (P \cup \{y\})$. We claim that $Q'$ is an independent set. Indeed, by Claim~\ref{claim:4-cycle} and that $G$ has no triangles, the subgraph induced by $N_G(u_1)\cup N_G(u_2)$ contains at most one edge, namely $xy$, but $y$ was deleted in the construction of $Q'$. Moreover, by Claim~\ref{claim-2-chord}, there are no edges between $(N_G(u_1)\cup N_G(u_2))\setminus P$ and $Q$ (note that $\{u_0,u_3\}\subseteq P\cup Q$).  
			
			Set $z' = y$ if $y\notin P$, and otherwise let $z'$ be an arbitrary vertex of $V(B_{G'})\setminus (P\cup Q')$.
			By the minimality of $G$, the graph $G'$ admits a coloring $\phi'$ 
			valid for $(G',P,Q',z',\eta)$. Extend $\phi'$ to a coloring $\phi$ of $G$ by setting 
			$\phi(u_1)=\phi(u_2)=1$. If $z=u_1$, then \ref{C5} holds since $\phi(z)=1$ and $u_2$ is the unique neighbor of $z$ colored $1$. If $z\neq u_1$ (and recall $z\neq u_2$), then $z$ is a neighbor of some 
			$v\in P$ with $\eta(v)=1$ by the assumption. Hence $\phi(z)=\phi'(z)=2$ by \ref{C3} for $\phi'$, so \ref{C5} also holds. It is routine to verify that $\phi$ satisfies \ref{C1}--\ref{C4} as well, and is therefore valid for $(G,P,Q,z,\eta)$, a contradiction. 
		\end{proof}
		
		In the sequel, we may assume without loss of generality that $z=u_2$.
		
		\begin{claim}\label{z-no-Q-neighbor}
			$\{u_1, u_3\} \cap Q = \emptyset$, that is, $z$ has no boundary neighbors in $Q$.
		\end{claim}  
		\begin{proof}  
			Without loss of generality, assume to the contrary that $u_3\in Q$. 
			By Claim~\ref{claim-QBQ}, we have $u_1\notin Q\cup\{v\in P \mid\eta(v)=2\}$. 
			Hence either $u_1\notin P\cup Q$, or $u_1\in P$ with $\eta(u_1)=1$. 
			
			If $u_1\notin P\cup Q$, then by the maximality of $|Q|$ it follows that $u_0\in Q$, 
			since otherwise we could add $u_1$ to $Q$. 
			This, however, contradicts Claim~\ref{claim-QBBQ} with $i=1$.  
			
			Therefore, we must have $u_1\in P$ with $\eta(u_1)=1$.
			As $u_3\in Q$ and $Q$ is independent, we have $u_4\notin Q$. 
			Since $k\ge 4$, $u_4\neq u_0$.
			
			If $u_4\in P$, then it must be that $k=4$ and $P=\{u_0,u_1,u_4\}$. Thus $\eta(u_4)=\eta(u_0)=2$, $Q= \{u_3\}$. Let $P'=\{u_3,u_4\}$, $Q'= N_{G}(u_1)$  and $\eta'$ be a precoloring of $P'$ with $\eta'(u_3) = \eta'(u_4) =2$. We also choose $z'\in V(B_{G'})\setminus (P'\cup Q')$. By the minimality of $G$, there a coloring $\phi'$ valid for $(G-\{u_1\}, P',  Q', z', \eta'\})$. 
			Extending $\phi'$ to $G$ by coloring $u_1$ with $1$, we obtain a valid coloring of $(G,P,Q,z,\eta)$, a contradiction. 
			
			Therefore, we conclude that $u_4\notin P$. 
			
			By Claim~\ref{claim-QBQ}, $u_5\notin Q\cup\{v\in P \mid \eta(v)=2\}$ 
			(note that indices are taken modulo $k+1$, so if $k=4$ then $u_5=u_0$). 
			Thus either $u_5\notin P$ or $u_5\in P$ with $\eta(u_5)=1$. 
			
			If $u_5 \notin P$, then by the maximality of $|Q|$ we have $u_6 \in Q$, 
			since otherwise we could add $u_6$ to $Q$. 
			This, however, contradicts Claim~\ref{claim-QBBQ} with $i=4$.

			We now have $u_1,u_5\in P$ with $\eta(u_1)=\eta(u_5)=1$. 
			Since $k\ge 4$, $u_5 \neq u_1$. 
			On the other hand, if $|P|=3$, then the two ends of $G[P]$ must receive different colors under $\eta$. 
			Thus it must be that $k=4$ (hence $u_5=u_0$) and $P=\{u_0,u_1\}$ with $\eta(u_0)=\eta(u_1)=1$.  
			In this case, let $G'=G-\{u_1\}$ and set $Q'=N_G(u_1)\setminus\{u_0\}$. 
			Define $P'=\{u_0,u_3,u_4\}$ with $\eta'(u_0)=1$ and $\eta'(u_3)=\eta'(u_4)=2$. 
			Choose $z'\in V(B_{G'})\setminus (P'\cup Q')$. 
			By the minimality of $G$, $G'$ admits a coloring $\phi$ valid for $(G',P',Q',z',\eta')$. 
			Extending $\phi$ to $G$ by coloring $u_1$ with $1$, we obtain a valid coloring of $(G,P,Q,z,\eta)$, a contradiction. 
		\end{proof} 
		
		Recall that $u_2=z\notin P\cup Q$. By Claim~\ref{claim:BBB}, we have $\{u_1,u_3\}\cap (P\cup Q)\neq\emptyset$, and hence by Claim~\ref{z-no-Q-neighbor}, $\{u_1,u_3\}\cap P\neq\emptyset$. 
		Without loss of generality, assume $u_1\in P$. Since $u_3\notin Q$, the maximality of $|Q|$ implies that $u_4\in Q$, for otherwise we could add $u_3$ to $Q$. Let $G' = G - \{u_3\}$ and define $Q' = (Q \cup N_G(u_3)) \setminus P$. By Claims~\ref{claim-chord},~\ref{claim-2-chord} and 
		the facts that  $u_1\notin Q$ and $N_G(u_3)$ is an independent set, 
		it follows that $Q'$ is also an independent set. 
		Let $z'$ be an arbitrary vertex of $V(B_{G'})\setminus (P\cup Q')$. 
		By the minimality of $G$, the graph $G'$ admits a coloring $\phi'$ valid for $(G',P,Q',z',\eta)$. 
		We then extend $\phi'$ to a coloring $\phi$ of $G$ by setting $\phi(u_3)=1$. Since this adds only a single isolated vertex to the induced forest of color~1 and $\phi(z)=2$, the coloring $\phi$ is valid for $(G,P,Q,z,\eta)$, a contradiction.  
		
		This completes the proof of Theorem~\ref{main-thm}. 
	\end{proof}
	
	\bibliographystyle{abbrv}
	\bibliography{reference}
	
\end{document}